\documentclass[12pt,reqno]{amsart}

\usepackage{hyperref}

\usepackage[all]{xy}

\xymatrixcolsep{2cm}

\numberwithin{equation}{section}

\allowdisplaybreaks[1]

\usepackage{etoolbox}

\makeatletter

\patchcmd{\thesubsection}{\arabic}{\Alph}{}{}

\patchcmd{\@seccntformat}{\@secnumfont}{%
\@secnumfont\expandafter\protect\csname format#1\endcsname}{}{}

\patchcmd{\@startsection}{\@afterindenttrue}{\@afterindentfalse}{}{}

\patchcmd{\subsection}{-.5em}{.3\linespacing}{}{}

\makeatother

\theoremstyle{plain}

\newtheorem{theorem}{Theorem}[section]

\newtheorem{lemma}[theorem]{Lemma}

\newtheorem{corollary}[theorem]{Corollary}

\theoremstyle{remark}

\newcommand{\SKer}[1]{\ensuremath{\mathcal{K}er (#1)}}

\newcommand{\At}[1]{\ensuremath{\mathrm{At} (#1)}}

\newcommand{\Img}[1]{\ensuremath{\mathrm{Im} (#1)}}

\newcommand{\cat}[1]{\ensuremath{\mathcal{#1}}}

\newcommand{\END}[2][]{\ensuremath{\mathcal{E}\mathit{nd}_{#1} (#2)}}

\newcommand{\id}[1]{\ensuremath{\mathbf{1}_{#1}}}

\newcommand{\p}{\ensuremath{\mathbf{P}}}

\newcommand{\C}{\ensuremath{\mathbb{C}}}

\newcommand{\struct}[1]{\ensuremath{\mathcal{O}_{#1}}}

\newcommand{\DIFF}[4][]{\ensuremath{\mathcal{D}\mathit{iff}^{#1}_{#2}(#3,\,#4)}}

\newcommand{\Diff}[4][]{\ensuremath{\mathrm{Diff}^{#1}_{#2}(#3,\,#4)}}

\newcommand{\coh}[3]{\ensuremath{\mathrm{H}^{#1}(#2,\,#3)}}

\newcommand{\Aut}[2][]{\ensuremath{\mathrm{Aut}_{#1} (#2)}}

\begin{document}

\title[differential operators on Hitchin  variety]{ differential operators on Hitchin variety}

%
%

\author[A. Singh]{Anoop Singh}

\address{Harish-Chandra Research Institute, HBNI, Chhatnag Road, Jhusi,
Prayagraj 211019, India}

\email{anoopsingh@hri.res.in}

\subjclass[2010]{32C38, 14F10, 14D06, 53C05 }

\keywords{Differential operator, Hitchin variety, holomorphic connection.}

\begin{abstract}
We introduce the notion of Hitchin variety over $\C$. 
Let $L$ be a holomorphic line bundle over a Hitchin 
variety $X$. We investigate
the space of all global sections of sheaf of differential operators $\cat{D}^k (L)$ and symmetric powers of sheaf of 
first order differential operators $\cat{S}^k(\cat{D}^1 (L))$  over  $X$ and show that for a projective Hithcin variety both the spaces  are one dimensional. As an application, we show that the space $\cat{C}(L)$ of holomorphic connections  on  $L$  does not admit any non-constant regular function.
\end{abstract}

\maketitle

\section{Introduction}

Let $C$ be a smooth projective curve over $\C$ of genus $g \geq 2$.
Let \cat{M} := \cat{M}($r$, $d$) be the moduli space of stable vector bundles of rank
$r$ and degree $d$ on C, where $r$ and $d$ are coprime. Then \cat{M} is a smooth 
projective variety  of dimension $r^2(g-1)+1$. In \cite{H},  Hitchin proved that the cotangent bundle 
$T^*\cat{M}$ is an algebraically
completely integrable Hamiltonian system, that is, we get a Hitchin fibration (for more details see \cite{H}).
Let $\cat{M}_{Higgs}$ denote the moduli space of stable
Higgs bundles of rank $r$ and degree $d$ over $C$.
Then $T^*\cat{M}$ is an open subset of $\cat{M}_{Higgs}$
with codimension  of $\cat{M}_{Higgs} \setminus T^*\cat{M}$  in $\cat{M}_{Higgs}$ is at least $2$.
 Motivated by the properties of 
Hitchin fibration, we introduce the notion of 
{\bf Hitchin variety}, which is defined in  section \ref{sec:diff}.  We give one more example of Hitchin 
variety in section \ref{sec:diff}.

The study of space of global sections of sheaves over an 
algebraic variety is always an interesting but a difficult
problem. In \cite{B}, Biswas studied the global sections
of sheaves of differential operator on a polarised abelian variety. He showed that the space of global 
sections of sheaf of differential operators on an ample line bundle over an abelian variety is of dimension one.
Motivated by this and some specific properties of a
Hitchin variety, we shall study  the space of global 
sections of sheaf of differential operators on any holomorphic line bundle and
also on  
holomorphic vector bundle, using 
different techniques.

In \cite[Theorem 6.2, p.n. 110]{H}, Hitchin computed the
global sections of the symmetric powers of tangent bundle of the moduli space of stable vector bundles 
with fixed determinant, rank $2$ and odd degree over a
compact Riemann surface of genus $g \geq 2$.
 
In \cite[Lemma 4.1, p.n. 428]{B1}, the global sections
of symmetric powers of Atiyah algebra of a generalized 
theta line bundle $\Theta$ over the moduli space
$\cat{N}^s_C(L)$ of 
stable vector bundles of rank $r$ with fixed determinant 
$L$  of degree $d$, where $r$ and $d$ are coprime, have 
been studied. In \cite[Theorem 1.4]{AS}, a similar method has been 
used to compute the algebraic functions on the moduli space of logarithmic connections singular over a finite subset of a compact Riemann 
surface with fixed residues.

To the best of our knowledge, 
the global sections of sheaves of differential operators
on a holomorphic line bundle over the moduli space 
$\cat{M} = \cat{M}(r,d)$ have not been studied.

Let $X$ be a Hitchin variety (see the definition in the 
beginning of section \ref{sec:diff}) and $L$ be a 
holomorphic 
line bundle over $X$. For $k \geq 0$, we denote by 
$\cat{D}^k(L)$ the vector bundle over $X$ 
defined by the sheaf $\DIFF[k]{X}{L}{L}$ of differential operators on $L$.  In Theorem \ref{thm:1}, we show that
\begin{equation*}
\label{eq:1.1}
 \coh{0}{X}{\cat{S}^k(\cat{D}^1 (L))} \cong \coh{0}{X}{\struct{X}} 
\end{equation*}
for every $k \geq 0$, where $\cat{S}^k(\cat{D}^1 (L))$
denotes the $k$-th symmetric power of $\cat{D}^1(L)$.
Moreover, if $X$ is a projective Hitchin variety, then 
we have  (see Corollary \eqref{cor:1})
\begin{enumerate}
\item \label{eq:30.1} \coh{0}{X}{\cat{S}^k(\cat{D}^1 (L))} = \C
\item  \label{eq:31.1} \coh{0}{X}{\cat{D}^k (L)} = \C
\end{enumerate}

Also, if $E$ is a holomorphic vector bundle over $X$,
then in Theorem \ref{thm:2}, we state the same result
for vector bundle $\cat{B}(E)$ constructed in \eqref{eq:38.6}.
As a consequence of Theorem \ref{thm:2}, for a projective
Hitchin variety $X$, we get that (see Corollary \eqref{cor:2})
\begin{equation*}
\coh{0}{X}{\cat{S}^k(\cat{B}(E))} = \C.
\end{equation*}

In the last section \ref{conn}, we consider the variety $\cat{C}(L)$ of holomorphic connections on $L$,
and in Theorem \ref{thm:3}, we prove that 
$\coh{0}{\cat{C}(L)}{\struct{\cat{C}(L)}} = \C$.

\section{Differential operators on line bundles}
\label{sec:diff}
 A smooth algebraic variety $X$
over $\C$  of dimension $n$ is said to be {\bf Hitchin variety} if 
there exist  normal noetherian schemes $\cat{N}$, $V$
over $\C$, where $V$ is  of 
dimension $n$, and a 
surjective proper  morphism 
\begin{equation}
\label{eq:1}
\cat{H}: \cat{N} \to V
\end{equation}
such that the following properties are satisfied.
\begin{enumerate}
\item The cotangent bundle $T^*X$ is an open subset of $\cat{N}$ with codimension of the
compliment $\cat{N} \setminus T^*X$ in $\cat{N}$ is 
 at least $2$.
\item A {\bf generic fibre} of $\cat{H}$ is of the form $$\cat{H}^{-1}(v) = A_v$$
such that $$\cat{H}^{-1}(v) \cap T^*X = A_v \setminus F_v$$
where 
$v \in V,$ 
$A_v$ is some abelian variety and $F_v$ is a closed 
subvariety of 
$A_v$ with $\mbox{codim}(F_v,A_v) \geq 2$.
\end{enumerate}

Note that the tangent bundle of an abelian variety is 
trivial, hence the vector fields on generic fibres 
of $\cat{H}$ contained in $T^*X$ are constant, which follows from Hartogs' theorem.

Such a phenomenon occurs naturally in the study of 
moduli space of stable vector bundles 
over a compact Riemann surface, 
where the cotangent bundle is an algebraically
completely integrable Hamiltonian system,
and is contained as an open subset in the moduli space 
of stable Higgs bundles with the codimension of its 
compliment in the moduli space of stable Higgs bundles is 
at least $2$ (see  \cite{H}).
The property mentioned in the definition of Hitchin 
variety is known as algebraic complete integrability.
See \cite[Definition 1]{PB} for a general definition of finite-dimensional complex algebraic integrable Hamiltonian system.

Another example of Hitchin variety (see \cite[section 4, p.n. 1185]{B2}) is as follows.

Let $Y$ be a compact Riemann surface of genus 
$g \geq 3$. Let $G$ be a non-trivial connected 
semi-simple linear algebraic group over $\C$ with Lie 
algebra $\mathfrak{g}$. Let $\gamma \in \pi_{1}(G)$
be the topological type of a holomorphic principal
$G$-bundles $E_{G}$ over $Y$. Let $\cat{M}^{\gamma}(Y,G)$
denote the moduli space of semi-stable holomorphic 
principal $G$-bundle $E_G$ over $Y$ of topological 
type $\gamma$.
A stable principal $G$–bundle $E_G$ over $Y$ is called {\bf regularly stable} if the automorphism group $\Aut{E_G}$ is just the center of $G$. 
The regularly stable locus
$$\cat{M}^{\gamma,rs}(Y, G) \subseteq \cat{M}^{\gamma}(Y, G)$$
is open, and coincides with the smooth locus of $\cat{M}^{\gamma}(Y, G)$; see [\cite{B3}, Corollary 3.4].
Then the moduli space $\cat{M}^{\gamma,rs}(Y, G)$
is a Hitchin variety, because $\cat{M}^{\gamma,rs}(Y, G)$ is a smooth quasi-projective variety with
$T^*\cat{M}^{\gamma,rs}(Y, G) \subseteq \cat{M}^{\gamma,rs}_{Higgs}(Y, G),$
where $\cat{M}^{\gamma,rs}_{Higgs}(Y, G)$ denotes the moduli space of regularly stable Higgs $G$–
bundles $(E_G, \theta)$ over $Y$ with $E_G$ of topological type $\gamma \in  \pi_1(G)$, and 
 
$$
\cat{H}: \cat{M}^{\gamma,rs}_{Higgs}(Y, G) \longrightarrow  V=
\bigoplus_{i=1}^{rank(G)}  \coh{0}{Y}{K^{\otimes n_i}_Y} $$ is the Hitchin map,
where $n_i$ are the degree of generators for the algebra 
$Sym(\mathfrak{g}^*)^{G}$; see \cite[section 4]{H},\cite{L}.

We state a version of Zariski's main theorem,
which can be proved using Stein factorisation. We use the following lemma  
to prove Theorem \ref{thm:1}.

\begin{lemma}
\label{lemm:1}
Let $T$, $S$ be two noetherian schemes over $\C$.
Let $\pi : T \to S$ be a surjective proper morphism such
that the generic fibres are connected. Suppose that $S$ is normal. Then the natural morphism of sheaves $
\pi^{\sharp}: \struct{S}\longrightarrow \pi_* \struct{T},$
is an isomorphism and every fibre of $\pi$ is connected.
\end{lemma}

Let $L$ be a holomorphic line bundle over $X$.
We follow \cite{GD} and \cite{R} for the definition and 
properties of differential operators.

Let $k \geq 0$ be an integer.
 A differential operator of order $k$ is a $\C$-linear 
map
\begin{equation}
\label{eq:2}
P: L \to L
\end{equation}
such that for every open subset $U$ of $X$ and for every 
$f \in \struct{X}(U)$, the bracket 
$$[P|_U,f]:L|_U \to L|_U$$ defined as 
\begin{equation*}
\label{eq:67}
[P|_U,f]_V(s) = P_V(f|_V s) - f|_V P_V(s)
\end{equation*}
is a differential operator of order $k-1$, for every open subset $V$ of 
$U$, and for all $ s \in L(V)$,
where differential operator of order zero from $L$ to $L$ is just 
\struct{X}-module homomorphism.

Let $\Diff[k]{X}{L}{L}$ denote the set of all differential operator of order 
$k$.  Then 
$\Diff[k]{X}{L}{L}$ is an 
$\struct{X}(X)$-module. For every open subset $U$ of $X$, 
$U \mapsto \Diff[k]{X}{L|_U}{L|_U}$ is a sheaf of differential operator of 
order $k$ from $L|_U$ to 
$L|_U$. This sheaf is denoted by $\DIFF[k]{X}{L}{L}$, which is locally free.

For $k \geq 0$, we denote by $\cat{D}^k(L)$ the vector bundle over $X$ 
defined by the sheaf $\DIFF[k]{X}{L}{L}$. Note that $\cat{D}^0(L) = 
\struct{X}$ and we have following inclusion of vector bundles
\begin{equation}
\label{eq:3}
\struct{X} = \cat{D}^0(L) \subset \cdots \subset \cat{D}^k(L) \subset \cat{D}^{k+1}(L) \subset \cdots
\end{equation}                                                                        

We have exact sequence of vector bundles 
\begin{equation}
\label{eq:4}
0 \to \cat{D}^{k-1}(L) \to \cat{D}^{k}(L) \xrightarrow{\sigma_k} \cat{S}
^k(TX) \to 0,
\end{equation}
where $\cat{S}^k(TX)$ is the $k$-th symmetric power of the tangent bundle of $X$
 and $\sigma_k$ is the symbol homomorphism. 
It is called symbol exact sequence (see \cite{R}, Chapter 2, for more 
details).

Consider the first order symbol operator 
\begin{equation*}
\label{eq:5}
\sigma_1: \cat{D}^1 (L)\to   TX
\end{equation*}
given in symbol exact sequence \eqref{eq:4}.
This induces a morphism 
\begin{equation*}
\label{eq:6}
\cat{S}^k(\sigma_1): \cat{S}^k(\cat{D}^1 (L))  \to 
\cat{S}^k (TX) 
\end{equation*}
of $k$-th symmetric powers.
Now, because of the following composition 
\begin{equation*}
\cat{S}^{k-1} (\cat{D}^1 (L)) = \struct{X} \otimes \cat{S}^{k-1}(\cat{D}^1 (L)) \hookrightarrow 
\cat{D}^1 (L) \otimes \cat{S}^{k-1} (\cat{D}^1 (L)) \to 
\cat{S}^{k} (\cat{D}^1 (L)),
\end{equation*}
we have 
\begin{equation*}
\label{eq:7}
\cat{S}^{k-1}(\cat{D}^1 (L))  \subset \cat{S}^k (\cat{D}^1 (L))
~~~ \mbox{for all}~ k \geq 1.
\end{equation*}

Thus, we get a short exact sequence of vector bundles over
$X$,
\begin{equation}
\label{eq:8}
0 \to \cat{S}^{k-1}(\cat{D}^1 (L)) \to \cat{S}^{k}(\cat{D}^1 (L)) \xrightarrow{\cat{S}^k (\sigma_1)} \cat{S}^k (TX) \to 0.
\end{equation}
 
 In other words, we get a filtration 
 \begin{equation}
 \label{eq:9}
 0 \subset \cat{S}^0(\cat{D}^1 (L)) \subset \cat{S}^1(\cat{D}^1 (L)) \subset \ldots \subset \cat{S}^{k-1}(\cat{D}^1 (L)) \subset \cat{S}^k(\cat{D}^1 (L)) \subset \ldots
 \end{equation}

 such that 
 \begin{equation}
 \label{eq:10}
 \cat{S}^{k}(\cat{D}^1 (L)) / \cat{S}^{k-1}(\cat{D}^1 (L)) \cong \cat{S}^k (TX)~~~ \mbox{for all}~ k \geq 1.
 \end{equation}
 
 Above filtration in \eqref{eq:9} gives the following 
 increasing chain of $\C$-vector spaces
 
 \begin{equation}
 \label{eq:11}
 \coh{0}{X}{\struct{X}} \subset \coh{0}{X}{\cat{S}^1(\cat{D}^1 (L))}
 \subset \coh{0}{X}{\cat{S}^2(\cat{D}^1 (L))} \subset  \ldots
 \end{equation}

\begin{theorem}
\label{thm:1} Let $X$ be a Hitchin variety. Then,
for $k \geq 0$, the inclusion homomorphism 
\begin{equation}
\label{eq:12}
\coh{0}{X}{\struct{X}} \to \coh{0}{X}{\cat{S}^k(\cat{D}^1 (L))}
\end{equation}
is an isomorphism.
\end{theorem}
\begin{proof}
From isomorphism in \eqref{eq:10}, we have the following commutative diagram
\begin{equation}
\label{eq:cd1}
\xymatrix{
0 \ar[r] & \cat{S}^{k-1}(\cat{D}^1 (L)) \ar[d] \ar[r] & \cat{S}
^k(\cat{D}^1 (L))
\ar[d] \ar[r]^{\cat{S}^k(\sigma_1)} & \cat{S}^k(TX) \ar[d] \ar[r] & 0 \\
0 \ar[r] & \cat{S}^{k-1}(TX) \ar[r] & 
\frac{\cat{S}^k(\cat{D}^1 (L))}{\cat{S}^{k-2}(\cat{D}^1 (L))} \ar[r] & \cat{S}^k(TX) \ar[r] & 0 
}
\end{equation}

which gives rise to the following commutative 
diagram of long exact sequences
\begin{equation}
\label{eq:cd2}
\xymatrix{
\cdots \ar[r] & \coh{0}{X}{\cat{S}^k TX} \ar[d] \ar[r]^{\delta'_k} & \coh{1}
{X}{\cat{S}^{k-1}(\cat{D}^1 (L))} \ar[d] 
\ar[r] & \cdots  \\
\cdots \ar[r] & \coh{0}{X}{\cat{S}^k TX}       \ar[r]^{\delta_k} & \coh{1}
{X}{\cat{S}^{k-1} TX} 
\ar[r] & \cdots }
\end{equation}
Note that proving  \eqref{eq:12} is equivalent to show the following
\begin{equation*}
\label{eq:13}
\coh{0}{X}{\cat{S}^{k-1}(\cat{D}^1 (L))}
\cong \coh{0}{X}{\cat{S}^k(\cat{D}^1 (L))}
~~~ \mbox{for all}~ k \geq 1.
\end{equation*}

Now, from above commutative diagram \eqref{eq:cd2}, it is enough 
to show that 
the connecting homomorphism $\delta'_k$ is injective for all $k
\geq 1$, which is equivalent to showing that 
the connecting homomorphism 
\begin{equation}
\delta_k: \coh{0}{X}{\cat{S}^k TX}      \to  \coh{1}{X}
{\cat{S}^{k-1} TX}
\end{equation}
is injective for every $k \geq 1$.

Moreover, a connecting homomorphism can be expressed as the
cup product by the extension class of the corresponding
short exact  sequence.  We denote the extension class
of the following short exact sequence 
\begin{equation}
\label{eq:14}
0 \to  \cat{S}^{k-1}(TX) \to
\frac{\cat{S}^k(\cat{D}^1 (L))}{\cat{S}^{k-1}(\cat{D}^1 (L))} \to \cat{S}^k(TX) \to 0 
\end{equation}
by $\Gamma_k$.
   
Next, $\Gamma_k$ can be described in terms of 
first Chern class $c_1(L)$ of the line bundle $L$, because
the first Chern class of $c_1(L)$ is nothing  but
the extension class of the following {\bf Atiyah exact
sequence} (see \cite{A})

\begin{equation}
\label{eq:15}
0 \to \struct{X} \to \cat{D}^1 (L) \xrightarrow{ \sigma_1} TX \to 0,
\end{equation}
and the short exact sequence \eqref{eq:8} is 
the $k$-th symmetric power of the Atiyah exact sequence 
\eqref{eq:15}

Thus, the connecting homomorphism $\delta_k$ can be described using the first Chern
class $c_1(L) \in \coh{1}{X}
{ T^* X}$ of the line bundle $L$ over $X$. 
Indeed, 
the cup 
product with $ c_1(L)$ gives rise to a homomorphism
\begin{equation}
\label{eq:16}
\mu: \coh{0}{X}{\cat{S}^k TX} \to \coh{1}{X}{\cat{S}^k TX \otimes T^*X}
\end{equation}
Also, we have a canonical homomorphism of vector bundles
\begin{equation*}
\label{eq:17}
\beta:\cat{S}^k TX \otimes T^*X \to  \cat{S}^{k-1}TX
\end{equation*}
which induces a morphism of  \C-vector spaces
\begin{equation}
\label{eq:18}
\beta^*:\coh{1}{X}{\cat{S}^k TX \otimes T^*X} \to \coh{1}
{X}{\cat{S}^{k-1} TX}.
\end{equation}
So, we get a morphism
\begin{equation}
\label{eq:19}
\tilde{\mu} = \beta^* \circ \mu: \coh{0}{X}{\cat{S}^k T 
X} \to \coh{1}{X}
{\cat{S}^{k-1} TX}.
\end{equation}
Then from the above observation we have $\tilde{\mu} = \delta_k$. 
It is sufficient to show that $\tilde{\mu}$ is injective.

Moreover, we have the  natural projection
\begin{equation}
\label{eq:20}
\eta: T^*X \to X
\end{equation}
and 
\begin{equation}
\label{eq:21}
\eta_* \eta^* \struct{X} = \oplus_{k \geq 0} \cat{S}^k TX.
\end{equation}
Thus, we have 
\begin{equation}
\label{eq:22}
\coh{j}{T^*X}
{\struct{ T^*X}} = \oplus_{k \geq 0}
\coh{j}{X}
{\cat{S}^{k} TX} ~~~ \mbox{for all}~ j \geq 0.
\end{equation}

To compute $\coh{j}{T^*X}{\struct{ T^*X}}$, we will use 
the fact that $X$ is a Hitchin variety. 

Let $g: T^*X \to \C$ be an algebraic function. 
Since the codimension of $\cat{N} \setminus T^*X$ in
$\cat{N}$ is at least $2$, and $\cat{N}$ is normal
noetherian scheme, $g$ extends by Hartog's
theorem to an algebraic function $\hat{g}$ on $\cat{N}$.
Since $\cat{H}$ is surjective proper morphism,
$V$ is normal, and generic fibres of $\cat{H}$ are 
abelian varieties, from Lemma \ref{lemm:1}
we have an isomorphism $$\struct{V} \cong \cat{H}_*\struct{ \cat{N}},$$ and hence there exists a unique algebraic
function
$\tilde{g} : V \to \C$ such that $$\hat{g} = \tilde{g}
\circ \cat{H}.$$

Set $\cat{V} = \mbox{d}( \coh{0}{V}
{\struct{V}}) \subset  \coh{0}{V}
{\Omega^1_{V}}$ the space of all exact algebraic $1$-form.
Define a map 
\begin{equation}
\label{eq:24}
\theta: \coh{0}{T^*X}
{\struct{ T^*X}} \to \cat{V}
\end{equation}
by $g \mapsto d \tilde{g}$, where $\tilde{g}$ is the function
on $V$ which is defined by descent of $g$ as above.  Then 
$\theta$ is an isomorphism.

From \eqref{eq:22} and \eqref{eq:24}, we have
\begin{equation}
\label{eq:25}
\theta:\oplus_{k \geq 0}
\coh{0}{X}
{\cat{S}^{k} TX} \to \cat{V}
\end{equation}
which is an isomorphism.

Now, restrict $\cat{H}$ on $T^*X$, and 
let $T_\cat{H} = T_{T^*X / V} = \SKer{d \cat{H}}$
be the relative tangent sheaf on $T^*X$,
where $$d \cat{H}: T(T^*X) \to \cat{H}^*TV$$ is a morphism 
of bundles.
Now, we use the fact that the vector fields on
$\cat{H}^{-1}(v)$ are constant, therefore, the pulled 
back bundle $\cat{H}^*T^*V$ is identified with $T_\cat{H}$, therefore, we have
$$ \coh{0}{V}{\Omega^1_{V}} \subseteq  \coh{0}{T^*X}
{T_\cat{H}}, $$ and hence from \eqref{eq:25}, we have
an injective homomorphism
\begin{equation}
\label{eq:26}
\nu: \cat{V} = \oplus_{k \geq 0} \theta(
\coh{0}{X}
{\cat{S}^{k} TX}) \to \coh{0}{T^*X}{T_\cat{H}}.
\end{equation}

Consider the morphism 
$$\coh{0}{T^*X}{T_\cat{H}} \to \coh{1}{T^*X}{T_\cat{H} \otimes T^* T^* X}$$ defined 
by taking cup product with the first Chern class $$c_1(\eta^* L) \in  \coh{1}{T^*X}{ T^* T^* X}.$$ Further
using the pairing $$T_\cat{H} \otimes T^* T^*X \to \struct{T^*X},$$ we get a homomorphism
\begin{equation}
\label{eq:27}
\psi: \coh{0}{T^*X}{T_\cat{H}} \to \coh{1}{T^*X}{\struct{T^*X}}.
\end{equation}

Since $c_1(\eta^* L) = \eta^*(c_1 L)$, we have
\begin{equation}
\label{eq:28}
\psi \circ \nu \circ \theta(\omega_k) = \tilde{\mu}(\omega_k),
\end{equation}
for all $\omega_k \in \coh{0}{X}
{\cat{S}^{k} T X}$, because of \eqref{eq:22}.
Since $\nu$ and $\theta$ are injective homomorphisms, it 
is enough to 
show that $\psi|_{\nu(\cat{V})}$ is injective 
homomorphism. 
Let $\omega \in \cat{V} \setminus \{0\}$ be a non-zero
exact $1$-form. Choose $u \in V$ such that $\omega(u) 
\neq 0$. As previously discussed $$\cat{H}^{-1}(u)
\cap T^*X = A_u 
\setminus F_u,$$ where $A_u$ is an abelian variety and 
$F_u$ is 
a subvariety of $A_u$ such that $\mbox{codim}(F_u,A_u) \geq 
2$.
Now, $\psi(\nu(\omega)) \in \coh{1}{T^*X}
{\struct{T^*X}}$ and we have the restriction 
map $$\coh{1}{T^*X}{\struct{T^*X}} \to \coh{1}{\cat{H}^{-1}(u)\cap T^*X}{\struct{\cat{H}^{-1}(u) \cap T^*X}}.$$
Since $ \omega(u) \neq 0$, $\psi(\nu(\omega)) \in 
\coh{1}{\cat{H}^{-1}(u)\cap T^*X}{\struct{\cat{H}^{-1}(u) \cap T^*X}}$.
Because of the following isomorphisms
\begin{equation*}
\label{eq:29}
 \coh{1}{\cat{H}^{-1}(u)\cap T^*X}{\struct{\cat{H}^{-1}(u) \cap T^*X}} \cong  \coh{1}{A_u}{\struct{A_u}} \cong  \coh{0}{A_u}{TA_u},
\end{equation*}
it follows that $\psi(\nu(\omega)) \neq 0$.
This completes the proof.

\end{proof}

\begin{corollary}
\label{cor:1}
Let $X$ be a projective Hitchin variety. Then, for
every line bundle $L$ on $X$ and 
for every $k \geq 0$, we have 

\begin{enumerate}
\item \label{eq:30} \coh{0}{X}{\cat{S}^k(\cat{D}^1 (L))} = \C
\item  \label{eq:31} \coh{0}{X}{\cat{D}^k (L)} = \C
\end{enumerate}

\end{corollary}
\begin{proof} 
\eqref{eq:30} follows immediately from Theorem \ref{thm:1}.

For \eqref{eq:31}, we need to show that 
the inclusions of $\C$-vector spaces
 \begin{equation}
 \label{eq:32}
 \coh{0}{X}{\struct{X}} \subset \coh{0}{X}{\cat{D}^1 (L)}
 \subset \coh{0}{X}{\cat{D}^k (L)} \subset  \ldots
 \end{equation}
 induced from \eqref{eq:3}, are equal, that is,
\begin{equation}
\label{eq:33}
\coh{0}{X}{\cat{D}^{k-1} (L)}
\cong \coh{0}{X}{\cat{D}^k (L))}
~~~ \mbox{for all}~ k \geq 1.
\end{equation}

For that consider the following commutative 
diagram
\begin{equation}
\label{eq:cd3}
\xymatrix{
0 \ar[r] & \cat{D}^{k-1} (L) \ar[d] \ar[r] & \cat{D}^k (L)
\ar[d] \ar[r]^{\sigma_k} & \cat{S}^k(TX) \ar[d] \ar[r] & 0 \\
0 \ar[r] & \cat{S}^{k-1}(TX) \ar[r] & 
\frac{\cat{D}^k (L)}{\cat{D}^{k-2} (L)} \ar[r] & \cat{S}^k(TX) \ar[r] & 0 
}
\end{equation}
where the top short exact sequence is symbol exact sequence \eqref{eq:4}.
The above commutative diagram gives the following commutative diagram of long exact sequences
\begin{equation}
\label{eq:cd4}
\xymatrix{
\cdots \ar[r] & \coh{0}{X}{\cat{S}^k TX} \ar[d] \ar[r]^{\beta_k} & \coh{1}
{X}{\cat{D}^{k-1} (L)} \ar[d] 
\ar[r] & \cdots  \\
\cdots \ar[r] & \coh{0}{X}{\cat{S}^k TX}       \ar[r]^{\delta_k} & \coh{1}
{X}{\cat{S}^{k-1} TX} 
\ar[r] & \cdots }
\end{equation}
To show \eqref{eq:33}, it is enough to show that  the connecting homomorphism $\beta_k$ is injective for all $k \geq 1$, which is 
equivalent to showing that $\delta_k$ is injective 
for all $k \geq 1$. 
Note that $\delta_k$ is the same connecting homomorphism
arises in the proof of the above Theorem \ref{thm:1},
which is injective. This completes the proof of 
\eqref{eq:31}.
\end{proof}

Let $E$ be a holomorphic vector bundle over a 
Hitchin variety  $X$.
We denote by $\cat{D}^k(E)$ the vector bundle over $X$ associated to  the sheaf $\DIFF[k]{X}{E}{E}$ of $k$-th 
order differential operators on $E$.
We have the {\bf symbol exact sequence} of $E$. 
\begin{equation}
\label{eq:37.1}
0 \to \END{E} 
\xrightarrow{\imath} \cat{D}^1(E)
\xrightarrow{\sigma_1}  TX \otimes 
{\END{E}} \to 0,
\end{equation}
which further reduces to 
the
{\bf Atiyah exact
sequence} (see \cite{A}) of $E$
\begin{equation}
\label{eq:38.1}
0 \to \END{E} \xrightarrow{\imath} \At{E} \xrightarrow{ 
\sigma_1} TX \to 0.
\end{equation}

Tensoring \eqref{eq:38.1} with $\Omega^1_X$ produces a 
short exact sequence

\begin{equation}
\label{eq:38.2}
0 \to \Omega^1_X \otimes \END{E} \xrightarrow{\imath}  \Omega^1_X \otimes \At{E} \xrightarrow{ 
\id{\Omega^1_X} \otimes \sigma_1} \Omega^1_X \otimes TX \to 0.
\end{equation}

Note that $\struct{X}\cdot \text{Id} \, \subset\, \END{TX}\,=\,
\Omega^1_{X}\otimes TX$. Define
$$
\Omega^1_{X}(\cat{A}t'(E))\, :=\, ({\id{\Omega^1_X} \otimes \sigma_1})^{-1}(\struct{X}\cdot \text{Id})\, \subset\,
\Omega^1_{X}(\At{E})\, ,
$$
where $\id{\Omega^1_X} \otimes \sigma_1$ is the projection in \eqref{eq:38.2}. So we have the short exact sequence of sheaves
\begin{equation}\label{eq:38.3}
0\,\longrightarrow\, \Omega^1_{X}(\END{E})\,\longrightarrow\, \Omega^1_{X}(\cat{A}t'(E))
\,\stackrel{\id{\Omega^1_X} \otimes \sigma_1}{\longrightarrow}\,\struct{X}\,\longrightarrow\, 0
\end{equation}
on $X$, where $\Omega^1_{X}(\cat{A}t'(E))$ is constructed above.

For simplicity, we shall denote $\id{\Omega^1_X} \otimes \sigma_1$ by $q$. Now, dualizing the above sequence \eqref{eq:38.3},
we get 
\begin{equation}
\label{eq:38.4}
0\,\longrightarrow\, \struct{X}\, \stackrel{q^*} {\longrightarrow}\, (\Omega^1_{X}(\cat{A}t'(E)))^*
\,{\longrightarrow}\, (\Omega^1_{X}(\END{E}))^*\,\longrightarrow\, 0
\end{equation} 
which is nothing but the following short exact sequence
\begin{equation}
\label{eq:38.5}
0\,\longrightarrow\, \struct{X}\, \stackrel{q^*} {\longrightarrow}\, T{X} \otimes (\cat{A}t'(E))^*
\,\stackrel{\eta} {\longrightarrow}\, T{X}\otimes (\END{E})^*\,\longrightarrow\, 0.
\end{equation}

Further define a sheaf
\begin{equation}
\label{eq:38.6}
\cat{B}(E) = \eta^{-1}(TX)
\end{equation}
Thus, we get the following exact sequence
\begin{equation}
\label{eq:38.7}
0\,\longrightarrow\, \struct{X}\, \stackrel{q^*} {\longrightarrow}\, \cat{B}(E)
\,\stackrel{\eta} {\longrightarrow}\, T{X}\,\longrightarrow\, 0
\end{equation}

which is similar to  \eqref{eq:15}. Now
for the vector bundle $\cat{B}(E)$ associated with $E$,
we have the following theorem which
admits an identical proof to the Theorem \ref{thm:1}.

\begin{theorem}
\label{thm:2}
For $k \geq 1$, and for every holomorphic vector bundle
$E$ over the Hitchin variety $X$, we have 
\begin{equation}
\label{eq:34}
\coh{0}{X}{\struct{X}} \cong \coh{0}{X}{\cat{S}^k(\cat{B}(E))}.
\end{equation}
\end{theorem}

 An immediate corollary of the above Theorem \ref{thm:2} is as follows
\begin{corollary}
\label{cor:2}
Let $X$ be a projective Hitchin variety. Then, for
every holomorphic  vector bundle $E$ on $X$  and 
for every $k \geq 1$, we have 
\begin{equation}
\label{eq:36} \coh{0}{X}{\cat{S}^k(\cat{B}(E))} = \C.
\end{equation}
\end{corollary}

\section{Sheaf of holomorphic connections on a holomorphic line bundle}
\label{conn}
As above, let $X$ be a  Hitchin variety and $L$
be a holomorphic line bundle over $X$.

Take the dual of the Atiyah exact sequence 
\eqref{eq:15} 
\begin{equation}
\label{eq:39}
0 \to \Omega^1_{X} 
\xrightarrow{\sigma_1^*} 
\cat{D}^1(L)^* \xrightarrow{\imath^*}   \struct{X} \to 0.
\end{equation}

Consider $\struct{X}$ as trivial line 
bundle $X \times \C$. Let $$s: X \to X \times \C$$ be a
holomorphic section of the trivial line bundle defined by $x \mapsto (x,1)$.

Let $S = \Img{s} \subset X \times \C$ be the image of $s$. 
Consider the inverse image $${\imath^*}^{-1}S \subset 
\cat{D}^1(L)^*,$$ and denote it by $\cat{C}(L)$.
Then for every open subset $U \subset X 
$, a holomorphic section of $\cat{C}(L)|_{U}$ over $U$ gives a holomorphic
splitting of the Atiyah exact sequence \eqref{eq:15}, 
associated to the holomorphic vector bundle $L|
_{U} \to U$. 
 For instance, suppose $\tau: U \to \cat{C}(L)|
_{U}$ is a holomorphic section. Then $\tau$ will be a 
holomorphic section of $\cat{D}^1(L)^*|_{U}$  over $U$, 
because $ \cat{C}(L) =
{\imath^*}^{-1}S \subset \cat{D}^1(L)^*$. Since $ 
 \tau \circ \imath =  \imath^*(\tau) = \id{U}$, so we 
 get a holomorphic  splitting $\tau$ of the Atiyah 
 exact sequence \eqref{eq:15}
 associated to $L|_{U}.$ Thus, $L|_{U}$ admits 
 a holomorphic connection. 
 Conversely, given any splitting of Atiyah exact sequence 
 \eqref{eq:15} over an open subset $U \subset X$, we get a holomorphic section of $\cat{C}(L)|_{U}$.
 
\begin{theorem}
\label{thm:3}
Let $X$ be a projective  Hitchin variety. Then, for
every holomorphic line bundle $L$, we have 
$$\coh{0}{\cat{C}(L)}{\struct{\cat{C}(L)}} = \C.$$
\end{theorem} 

\begin{proof}
Let 
$\p(\cat{D}^1(L))$ be the projectivization of 
$\cat{D}^1(L)$, that is, $\p(\cat{D}^1(L))$ parametrises hyperplanes in $\cat{D}^1(L)$.
Let $\p{(TX)}$ be the projectivization of the tangent bundle $TX$. 
Notice that $\p(TX)$ is a subvariety of 
$\p (\cat{D}^1(L))$, and  $\p(TX)$ is the zero
locus of the of a section of the tautological line bundle
$\struct{\p (\cat{D}^1(L))}(1)$. Now, observe that
$\cat{C}(L) = \p(\cat{D}^1(L)) \setminus \p(TX)$. Then we have

\begin{equation}
\label{eq:40}
\coh{0}{\cat{C}(L)}{\struct{\cat{C}(L)}} =
\varinjlim_{k} \coh{0}{\p (\cat{D}^1(L))}{\struct{\p (\cat{D}^1(L))}(k)} = \varinjlim_{k} \coh{0}{X}{\cat{S}^k\cat{D}^1(L)}
\end{equation}
\end{proof}
Now, from Corollary \eqref{cor:1} \eqref{eq:30}, we have the 
result.


\end{document}